\RequirePackage{etoolbox}
\csdef{input@path}{%
 {sty/},
 {img/},
 {bib/}
}

\documentclass[numbers,compress]{vmsta}

\volume{0}
\issue{0}
\pubyear{2020}
\articletype{research-article}

\newtheorem{theorem}{Theorem}
\newtheorem{remark}{Remark}
\newtheorem{lemma}{Lemma}

\theoremstyle{definition}

\newtheorem{assump}{Assumption}
\newtheorem*{remark*}{Remark}

\usepackage{mathtools}
\usepackage[nameinlink, noabbrev]{cleveref}
\crefname{assump}{assumption}{assumptions}
\newcommand{\verteq}{\rotatebox{90}{$\,=$}}
\newcommand{\vertleq}{\rotatebox{90}{$\,\leq$}}

\usepackage{mathsemantics}

\begin{document}
\begin{frontmatter}

\pretitle{Research Article}

\title{Minimax identity with robust utility functional for a non-concave utility}
\author[a]{\inits{O.}\fnms{Olena}~\snm{Bahchedjioglou}
	\thanksref{cor1}
	\ead[label=e1]{olenabahchedjioglou@gmail.com}\orcid{0000-0002-4801-8766}}
\thankstext[type=corresp,id=cor1]{Corresponding author.}

\author[b]{\inits{G.}\fnms{Georgiy}~\snm{Shevchenko}
	\ead[label=e2]{gshevchenko@kse.org.ua}}

\address[a]{\institution{Taras Shevchenko National University of Kyiv} \\
	Volodymyrska str., 01033 Kyiv, \cny{Ukraine}}

\address[b]{Kyiv School of Economics, 3 Mykoly Shpaka,\\
	03113 Kyiv, Ukraine}


\markboth{O. Bahchedjioglou, G. Shevchenko}{Minimax identity with robust utility functional for a non-concave utility}

\begin{abstract}
	We study the minimax identity for a non-decreasing upper-semicontinuous utility function satisfying mild growth assumption. In contrast to the classical setting, we do not impose the assumption that the utility function is concave. 
	By considering the concave envelope of the utility function we obtain equalities and inequalities between the robust utility functionals of an initial utility function and its concavification. Furthermore, we prove similar equalities and inequalities in the case of implementing an upper bound on the final endowment of the initial model.
\end{abstract}



\end{frontmatter}

\section{Introduction}
Consider a complete market model framework with unique equivalent local martingale measure $Q^e$.
In the spirit of Reichlin \cite{Reichlin11}, we consider a utility function $U$ on $\mathbb{R}_+$ which is non-decreasing upper-semicontinuous and satisfying a mild growth condition.  Schied and Wu \cite{Schied-Wu} impose the assumptions below on the set of probability measures $\mathcal{Q}$ on $(\Omega, \mathcal{F})$; note that $\mathcal{Q}$ is not the set of all measures on the measurable space $(\Omega, \mathcal{F}),$ but just a subset satisfying these assumptions.
\begin{assump}\label{ass1}
	\begin{itemize}
		\item[(i)] $\mathcal{Q}$ is convex;
		\item[(ii)] $\mathbb{P}[A]=0$ if and only if $Q[A]=0$ for all $Q \in \mathcal{Q};$
		\item[(iii)] The set $\mathcal{Z}:=\{dQ/d\mathbb{P}|Q\in \mathcal{Q}\}$ is closed in $L^{0}(\mathbb{P}).$
	\end{itemize}
Also, to the \Cref{ass1} we add
\begin{itemize}
	\item[\textit{(iv)}] \textit{The set $ \mathcal{Z}_e:=\{dQ/dP|Q\in \mathcal{Q}_e\}$  is closed in $ L^{0}(\mathbb{P}),$}
\end{itemize}
where $\mathcal{Q}_e$ denotes the set of measures in $\mathcal{Q}$ that are equivalent to $\mathbb{P}$.
\end{assump}

In this paper we study the minimax identity for the robust non-concave utility functional in a complete market model, i.e.
\begin{gather*}
u(x):=\sup\limits_{X \in \mathcal{X}(x)}\inf\limits_{Q \in \mathcal{Q}}E_{Q}[U(X)]=\inf\limits_{Q \in \mathcal{Q}}\sup\limits_{X \in \mathcal{X}(x)}E_{Q}[U(X)].
\end{gather*}
while considering two possibilities for the set $\mathcal{X}(x)$ of admissible final endowments: 
\begin{itemize}
	\item the standard budget constraint:
	$$\mathcal{X}(x) = \{X \in L^{1}_+(Q^e)| \mathrm{E}_{Q^e}[X]\le x\}, x>0,$$
	\item an additional upper bound: 
	\begin{equation}\label{upperbnd}
	\mathcal{X}^W(x) = \{X \in L^{1}(Q^e) \mid 0 \leq X \leq W,  \mathrm{E}_{Q^e}[X]\leq x\},
	\end{equation}
	with some random variable $W : \Omega \rightarrow [0, +\infty ).$
\end{itemize}

One of the key tasks of financial mathematics is proving the existence as well as the construction of the optimal investment strategies,  in other words, finding the utility-maximizing investment strategies.
Mostly, this problem was studied under the assumption that the probability measure which accurately describes value process development is known. 

However, in reality, not only the exact probabilities are unknown, but there are abundant aspects that can be considered in mentioned maximization problems such as the completeness of the market, the set of prior probability measures, the assumptions on investor's utility function, the modeling of payoff and so on. That is why instead of a single measure it is sound to consider the set of probability measures with natural assumptions on it. Thus, the standard utility maximization problem is transformed into the robust utility maximization, i.e.
\begin{gather*}
\sup\limits_{X \in \mathcal{X}(x)}\inf\limits_{Q \in \mathcal{Q}}\mathrm{E}_{Q}[U(X)],
\end{gather*} where one maximizes the expected utility under the infimum over the whole set of probability measures, for details see  Gilboa and Schmeidler \cite{Gilboa1987}, \cite{Gilboa2004}, \cite{Schmeidler1989}, Yaari \cite{yaari}, F\"{o}llmer and Schied \cite[Section
2.5]{Fol}.

In the case of a standard utility maximization problem it is possible to construct the optimal investment strategy given a strictly concave utility function, see F\"{o}llmer and Schied \cite[Section 2.5]{Fol}, and for the general case of utility functions see Bahchedjioglou and Shevchenko \cite{paper1}. Both references considered standard budget constraints as well as additional upper bound on the final endowments.  For a detailed survey of this problem in general model setup in both complete and incomplete market models but with risk-averse agent, see Biagini \cite{Biagini}.

In this paper, we consider the robust maximization problem with the general case of non-concave utility function likewise with and without budget constraints. 
In the previous literature different approaches were used for robust portfolio optimization such as reducing the robust case to the standard one through proving the existence of the ``worst-case scenario measure'' or ``the least favourable measure'', e.g. \cite{quenez, Schied}, a stochastic control approach, see \cite{Hernandez}, an approach using BSDEs, see \cite{Bordigoni} and references therein.

Besides, for solving the optimal investment problems one can make use of the following interim finding such as minimax identity and duality theory.
Using the minimax identity for the concave functions, see \cite[section 6]{Aubin84}, Schied and Wu \cite{Schied-Wu} showed the existence of optimal probability measure $\widehat{Q}$, in the sense that $$\sup\limits_{X \in \mathcal{X}(x)}\inf\limits_{Q \in \mathcal{Q}}\mathrm{E}_{Q}[U(X_T)]=
\sup\limits_{X \in \mathcal{X}(x)} \mathrm{E}_{\widehat{Q}}[U(X_T)],$$ which, together with the results of the Kramkov and Schachermayer \cite{Kramkov99,Kramkov03}, were the base for proving the existence of the optimal investment strategy. They used a general incomplete market model with rather natural assumptions on the set of probability measures.
Backhoff Veraguas and Fontbona \cite{backfont} extended these results by implementing the assumption on the densities of the uncertainty set instead of the usual compactness assumption. Moreover, they have done this without relying on the existence of the worst-case measure or on any assumption implying this. For more results concerning the robust utility maximization problem we refer to Bartl, Kupper and Neufeld \cite{Bartl2021} and references therein.

The majority of articles on utility maximization assume that the investor's utility function is strictly concave, strictly increasing, continuously differentiable, and satisfies the Inada conditions. While the assumption of monotonicity is natural, since an agent prefers more wealth to less, other assumptions can be omitted or relaxed. 
There is a wide class of models in which the non-concave and not necessarily continuously differentiable utility function maximization has been studied by reducing the problem to the concave case. One of the most important works was done by Reichlin \cite{PHD,Reichlin11}. He considered the general framework of the
non-concave utility functions for both complete and incomplete market models. By applying the concavification technique he established valuable relations between the maximization problems for a non-concave utility function $U$ and its concavification $U_c$ thereby reducing the task to the concave problem. Moreover, Reichlin proved the existence of the maximizer under certain assumptions and established its properties.

While considering two cases of admissible final endowments: the standard budget constraint and additional upper bound (which has not been considered before in such model setup) we extend Reichlin's results by proving new connections in the form of equalities and inequalities of the robust utility maximization functionals of initial non-concave utility functions and its concavification. Furthermore, we proceed in proving the minimax identity for the general case of non-concave utility functions. The crucial step for obtaining the mentioned results with implementing an additional upper bound is the use of the regular conditional distribution which sheds new light on the possible approaches for solving the optimization problem.

The paper is organized as follows. In \Cref{general} we study the minimax identity for a non-concave utility function in the complete market model. We do not prove nor refute the minimax identity, however, we show useful equalities and inequalities to relate the robust utility functional of initial utility functional and its concavification. \Cref{bound} is devoted to the study of the minimax identity under the implementation of budget constraints. The results of \Cref{bound} are similar to the corresponding results of \Cref{general}, however, some of proves differ significantly.

Throughout the paper the measurability of real-valued functions we will understand in the Borel sense.

\section{Minimax identity for non-concave utility functions in complete market model}\label{general}

This problem is already solved in \cite{Bul2019}, but, since we want to expand this problem by considering budget constraints we present the main part of the mentioned paper omitting the proofs.

\subsection{Formulation of the problem}
To formulate the goal of this paper first let us remind some notations.
For any initial capital $x>0$, let $\mathcal{X}(x)$ be the set of all possible random endowments corresponding to $x$, i.e.\ all random variables $X$  
such that $X \geq 0$, $\mathrm{E}_{Q^e}[X]\le x.$ 

Moreover, we consider a utility function $U$ which is non-decreasing, upper-semicontinuous, defined on a domain 
$(0,\infty)$ and satisfying the mild growth condition:
\begin{gather*}
\lim\limits_{x \rightarrow \infty}\frac{U(x)}{x}=0.
\end{gather*}

It follows from \cite[Proposition 3.1]{Aum} that $U(x)$ has a non-decreasing and continuous concave envelope $U_c(x)$, or the smallest concave function such that $U_c(x)\geq U(x)$ for all $x \in \mathbb{R}$; we will call it a \textit{concavification} of $U$.

This section aims is to prove some equalities and inequalities, related to the minimax identity for the robust non-concave utility functionals:
\begin{gather*}
\sup\limits_{X \in \mathcal{X}(x)}\inf\limits_{Q \in \mathcal{Q}}E_{Q}[U(X)]=\inf\limits_{Q \in \mathcal{Q}}\sup\limits_{X \in \mathcal{X}(x)}E_{Q}[U(X)].
\end{gather*}

We will assume that the probability space $(\Omega, \mathcal{F}, \mathbb{P})$ is atomless. 
Introduce the notation:
\begin{gather*}
u^c(x):=\sup\limits_{X \in \mathcal{X}(x)}\inf\limits_{Q \in \mathcal{Q}}E_{Q}[U_c(X)];\\
u_{Q}(x):=\sup\limits_{X \in \mathcal{X}(x)}E_{Q}[U(X)];\\
u_{Q}^c(x):=\sup\limits_{X \in \mathcal{X}(x)}E_{Q}[U_c(X)].
\end{gather*}

Also, we need the finiteness of value functions, which we can write as 
\begin{assump}\label{finiteu}
	\begin{gather*}
	\text{For all } x>0 \text{ exists a measure } Q_0\in \mathcal{Q}_e \text{ such that } u_{Q_0}(x)<\infty.
	\end{gather*}
\end{assump}
\begin{assump}\label[assump]{finit}
	\begin{gather*}
	u^{c}_{Q_0}(x)<\infty \text{ for some, and hence for all } x>0 \text{ and some } Q_0\in \mathcal{Q}_e.
	\end{gather*}
\end{assump}
Note, that finiteness of $u^c_Q(x)$ implies finiteness of $u_Q(x),$ since $u_Q(x)\leq u^c_Q(x).$

\begin{theorem}\label{mt}
	Suppose that Assumptions \ref{ass1}, \ref{finiteu}, \ref{finit} hold and that the probability space $(\Omega, \mathcal{F}, \mathbb{P})$ is atomless.
	
	Then the following holds
	
\small	
	\begin{tabular}{c@{\hskip 0.001cm}c@{}c@{\hskip 0.001cm}c@{}c}\hypertarget{star}{}
		$\sup\limits_{X \in \mathcal{X}(x)}\inf\limits_{Q \in\mathcal{Q}_e}E_{Q}[U_c(X)]$ & $\overset{(1\star)}{=}$ & $\sup\limits_{X \in \mathcal{X}(x)}\inf\limits_{Q \in \mathcal{Q}}E_{Q}[U_c(X)]$ & $\overset{(2\star)}{=}$ & $\inf\limits_{Q \in \mathcal{Q}}\sup\limits_{X \in \mathcal{X}(x)}E_{Q}[U_c(X)]$\\
		$\vertleq(4\star)$ & & & & $\verteq(3\star)$\\
		$\sup\limits_{X \in \mathcal{X}(x)}\inf\limits_{Q \in \mathcal{Q}_e}E_{Q}[U(X)] $ & & & &
		$\inf\limits_{Q \in \mathcal{Q}_e}\sup\limits_{X \in \mathcal{X}(x)}E_{Q}[U_c(X)]$\\
		$\verteq(6\star)$ & & & & 	$\verteq(5\star)$\\
		$\sup\limits_{X \in \mathcal{X}(x)}\inf\limits_{Q \in \mathcal{Q}}E_{Q}[U(X)]$ & $\overset{(7\star)}{\leq}$ & $\inf\limits_{Q \in \mathcal{Q}}\sup\limits_{X \in \mathcal{X}(x)}E_{Q}[U(X)]$ & $\overset{(8\star)}{\leq}$ &
		$\inf\limits_{Q \in \mathcal{Q}_e}\sup\limits_{X \in \mathcal{X}(x)}E_{Q}[U(X)]$
	\end{tabular}
\end{theorem}
\normalsize

The proof of this theorem will be divided into several parts.

\subsection{Minimax identity for the concavified objective function $U_c(x)$}

Now we are going to show that minimax identity holds for $U_c(x).$ 

There is a lot of literature with proofs of the minimax identity for the robust utility functionals, the most general case was considered in the \cite{Schied-Wu}. However, there the authors assume that the utility function is strictly increasing, strictly concave, and satisfies the Inada conditions both in point 0 and in  $\infty.$

The function $U_c(\cdot)$ which we are considering is no longer strictly concave and does not satisfy the Inada conditions in 0, hence we can not use all the previous results without changes. 

The next lemma is almost the same as \cite[Lemma 3.4]{Schied-Wu}.

\begin{lemma}\label[lemma]{Lemma4}
	Suppose that  \Cref{ass1} and \Cref{finit}  hold.
	
	Then, we have 
	\begin{gather*}
	u^{c}(x)=\sup\limits_{X \in \mathcal{X}(x)}\inf\limits_{Q \in \mathcal{Q}}E_{Q}[U_c(X)]=
	\inf\limits_{Q \in \mathcal{Q}}\sup\limits_{X \in \mathcal{X}(x)}E_{Q}[U_c(X)] \\
	=\sup\limits_{X \in \mathcal{X}(x)}\inf\limits_{Q \in \mathcal{Q}_e}E_{Q}[U_c(X)]=
	\inf\limits_{Q \in \mathcal{Q}_e}\sup\limits_{X \in \mathcal{X}(x)}E_{Q}[U_c(X)] 
	\end{gather*}
\end{lemma}
\begin{remark*}
	This lemma holds if we will consider utility function $U:(0,\infty)\rightarrow \mathbb{R}$ instead of $U:[0,\infty)\rightarrow \mathbb{R}$.
	Thus, we present the proof for a more general case.
\end{remark*}
\begin{proof}
	The proof can be found in \cite[Lemma 1]{Bul2019}.
\end{proof}

\subsection{Minimax identity for the objective function $U(x)$}

In this section, we want to prove lemmas which will help us to complete the proof of the \autoref{mt}.
\begin{remark*}
	Note, that the main argument in the proof of minimax identity for the robust utility maximization problem is the lop sided minimax theorem by Aubin and  Ekeland, see \cite[Chapter 6, p. 295]{Aubin84}, which holds if the
	functional $X \rightarrow \mathbb{E}[ZU(X)]$ is concave. Since we consider the non-concave utility function $U,$ we can not prove the minimax identity in this case similarly. A more general case of the minimax identity was proved by Maurice Sion, see \cite{[Sion]}. However, to use Sion's minimax theorem we still need 
	functional $X \rightarrow \mathbb{E}[ZU(X)]$ to be quasi-concave, which is not true, in the general case, even for the indicator functions multiplied  by the constants.
\end{remark*}
\begin{lemma}\label[lemma]{Lemma5}
	If \Cref{ass1} and \Cref{finiteu} hold,
	then for all $X \in \mathcal{X}(x)$ 
	\begin{gather}
	\inf\limits_{Q \in \mathcal{Q}}E_Q[U(X)] = 	\inf\limits_{Q \in \mathcal{Q}_e}E_Q[U(X)].
	\label{6star}
	\end{gather}
\end{lemma}	

\begin{proof}
	The proof can be found in \cite[Lemma 2]{Bul2019}.
\end{proof}

\begin{remark}
	The above lemma also holds for $U_c$ in place of $U$.
\end{remark}


The prove of equality \hyperlink{star}{$(5\star)$} is based on \cite[Theorem 5.1]{Reichlin11}. 

\begin{lemma}\label[lemma]{Lemma6}
	Suppose that $(\Omega, \mathcal{F}, \mathbb{P})$ is atomless. 
	
	Then it holds that
	\begin{gather*}
	\inf\limits_{Q \in \mathcal{Q}_e}\sup\limits_{X \in \mathcal{X}(x)}E_Q[U(X)] = \inf\limits_{Q \in \mathcal{Q}_e}\sup\limits_{X \in \mathcal{X}(x)}E_Q[U_c(X)].
	\label{infsup}
	\end{gather*}
\end{lemma}	

\begin{proof}
	The proof can be found in \cite[Lemma 3]{Bul2019}.
\end{proof}

\begin{proof}[Proof of the \autoref{mt}] \mbox{}\\*
	\begin{itemize}
		\item \hyperlink{star}{$(1\star)$} - \hyperlink{star}{$(3\star)$} follows from the \Cref{Lemma4};
		
		\item \hyperlink{star}{$(4\star)$} follows from the fact that $U_c \geq U$;
		
		\item \hyperlink{star}{$(5\star)$} follows from the \Cref{Lemma6};
		
		\item To obtain \hyperlink{star}{$(6\star)$} we need to take the $\sup\limits_{g \in C(x)}$ of the both sides in the equality \eqref{6star};
		
		\item The inequality \hyperlink{star}{$(7\star)$} follows from the fact that for all $Q \in \mathcal{Q} $ and all $X \in \mathcal{X}(x)$ holds 
		\begin{gather*}
		\inf\limits_{Q \in \mathcal{Q}}E_Q[U(X)] \leq \sup\limits_{X \in \mathcal{X}(x)}E_Q[U(X)].
		\end{gather*}
		
		\item Since $\mathcal{Q}_e  \subseteq \mathcal{Q}, $ the inequality \hyperlink{star}{$(8\star)$} is clear.
	\end{itemize}
\end{proof}
\section{Minimax identity for constrained case of random endowments}\label{bound}

\subsection{Formulation of the problem}

This section is in general similar to \Cref{general}, however, similarly to \cite[Chapter 3]{Fol} we consider a modified constrained counterpart. 

Specifically, we assume that there is an upper bound on the endowment, given by a random variable $W : \Omega \rightarrow (0, +\infty )$.
The set of admissible payoffs is then given by 
\begin{gather*}
\mathcal{X}^W := \{X \in L^{0}(\mathbb{P}) \mid 0 \leq X \leq W\  \mathbb{P}\text{-a.s.}\}
\end{gather*}
We keep all of the assumptions from \Cref{general} on the set of all probability measures $\mathcal{Q}$ and the utility function $U$ intact. For technical reasons we will also assume that $(\Omega,\mathcal F)$ is a standard Borel space, which in particular implies the existence of a regular conditional distribution given $W$. We will require that this conditional distribution is atomless, in other words, that the constraint $W$ leaves a sufficient amount of randomness.
\begin{assump}\label{atomless}
	\begin{enumerate}
		\item $(\Omega,\mathcal F)$ is a standard Borel space. 
		\item There exists a regular conditional distribution given $W$, which is atomless, i.e., there exists a function $P\colon \mathcal F\times (0,\infty)\to [0,\infty)$ such that for all $v>0$, $P(\cdot, v)$ is an atomless probability measure, and for all $A\in\mathcal F$, $P(A,\cdot)$ is a measurable function satisfying $P(A,W) = \mathbb{P}(A\mid W)$ a.s.
	\end{enumerate}
\end{assump}

As in \cite{paper1} for each $k>0$ denote
\begin{gather}
U^{k}(x)= U(x\land k),\ x \geq 0.
\end{gather}

Note that the function $U^{k}$ and its concavification $U^{k}_c$ satisfy all of the assumptions on the utility function $U$ and its concavification $U_c.$
%
Moreover, $U^{k}_c(x)=U^{k}(x),$ for all $x \geq k.$

Our goal is to prove some equalities and inequalities, related to the minimax identity for the robust non-concave utility functionals:
\begin{gather*}
\sup\limits_{ X \in \mathcal{X}^W_x}\inf\limits_{Q \in \mathcal{Q}}E_{Q}[U^{W(\omega)}(X)]=\inf\limits_{Q \in \mathcal{Q}}\sup\limits_{X \in \mathcal{X}^W_x}E_{Q}[U^{W(\omega)}(X)],
\end{gather*}
where the budget set $\mathcal{X}^W_x$ is defined by
\begin{gather*}
\mathcal{X}^W_x := \{X \in L^{1}(Q^e) \mid 0 \leq X \leq W,  E_{Q^e}[X]\leq x\},
\end{gather*}
where $x>0$ is the initial wealth and $Q^{e}$ is the unique equivalent local martingale measure.

Introduce the following notation:
\begin{gather*}
u^W_c(x):=\sup\limits_{X \in \mathcal{X}^W_x}\inf\limits_{Q \in \mathcal{Q}}E_{Q}[U_c^{W(\omega)}(X)];\\
u^W_{Q}(x):=\sup\limits_{X \in \mathcal{X}^W_x}E_{Q}[U^{W(\omega)}(X)];\\
u^W_{c,Q}(x):=\sup\limits_{X \in \mathcal{X}^W_x}E_{Q}[U_c^{W(\omega)}(X)].
\end{gather*} 

\begin{remark*}
	Since $U^k(x)\leq U(x), x\geq 0$ \Cref{finiteu} and \ref{finit} provide the finiteness of the value function above.
\end{remark*}

It is natural to consider only the case where
\begin{gather}
\mathrm{E}_{Q^e}[W] > x, \label{W}
\end{gather}
as otherwise, thanks to the monotonicity of $U$, the optimization problem has a trivial solution $X^* = W$. 

The formulation of the next theorems and lemmas are the same as in \Cref{general}. However, because of the boundness assumption on the endowments the proof of the corresponding statements will be different.

\begin{theorem}\label{mt2}
	Under Assumptions \ref{ass1},  \ref{finiteu}, \ref{finit}, \ref{atomless},
	we have the following: 	
	
\small
\begin{flushleft}	
	\begin{tabular}{c@{\hskip 0.001cm}c@{}c@{\hskip 0.001cm}c@{}c}\hypertarget{star2}{}
		$\sup\limits_{X \in \mathcal{X}^W_x}\inf\limits_{Q \in\mathcal{Q}_e}E_{Q}[U_c^{W(\omega)}(X)]$ & $\overset{(1\star)}{=}$ & $\sup\limits_{X \in \mathcal{X}^W_x}\inf\limits_{Q \in \mathcal{Q}}E_{Q}[U_c^{W(\omega)}(X)]$ & $\overset{(2\star)}{=}$ & $\inf\limits_{Q \in \mathcal{Q}}\sup\limits_{X \in \mathcal{X}^W_x}E_{Q}[U_c^{W(\omega)}(X)]$\\
		$\vertleq(4\star)$ & & & & $\verteq(3\star)$\\
		$\sup\limits_{X \in \mathcal{X}^W_x}\inf\limits_{Q \in \mathcal{Q}_e}E_{Q}[U^{W(\omega)}(X)] $ & & & &
		$\inf\limits_{Q \in \mathcal{Q}_e}\sup\limits_{X \in \mathcal{X}^W_x}E_{Q}[U_c^{W(\omega)}(X)]$\\
		$\verteq(6\star)$ & & & & 	$\verteq(5\star)$\\
		$\sup\limits_{X \in \mathcal{X}^W_x}\inf\limits_{Q \in \mathcal{Q}}E_{Q}[U^{W(\omega)}(X)]$ & $\overset{(7\star)}{\leq}$ & $\inf\limits_{Q \in \mathcal{Q}}\sup\limits_{X \in \mathcal{X}^W_x}E_{Q}[U^{W(\omega)}(X)]$ & $\overset{(8\star)}{\leq}$ &
		$\inf\limits_{Q \in \mathcal{Q}_e}\sup\limits_{X \in \mathcal{X}^W_x}E_{Q}[U^{W(\omega)}(X)]$
	\end{tabular}
\end{flushleft}	
\end{theorem}
\normalsize

The proof of this theorem will be divided into several parts.

\subsection{Minimax identity for the concavified objective function $U_c^{W(\omega)}(x)$}

Now we are going to show that minimax identity holds for $U_c^{W(\omega)}(x).$ First we prove some useful properties.
\begin{lemma}\label[lemma]{prop}
	\begin{itemize}
		\item[a)] Set $\mathcal{X}^W_{x}$ is convex.
		\item[b)]  It  holds that $u^W_{c}(x)= \sup\limits_{X \in \mathcal{X}^W_x}\inf\limits_{Q \in \mathcal{Q}}E_{Q}[U_c^{W(\omega)}(X)]$ is concave. 
	\end{itemize}
\end{lemma}
\begin{proof}
	\begin{itemize}
		\item [a)] Consider $X_1 \in \mathcal{X}^W_{x_1}$ and $X_2 \in \mathcal{X}^W_{x_{2}}$ for some $x_1, x_2 > 0$ and $\alpha \in (0,1).$ One has that $0\leq \alpha X_1+(1-\alpha)X_2 \leq W$ and $\mathrm{E}_{Q^e}[\alpha X_1+(1-\alpha)X_2] \leq \alpha x_1+(1-\alpha)x_2$. Hence, $\alpha X_1+(1-\alpha)X_2 \in \mathcal{X}^W_{\alpha x_1+(1-\alpha)x_2}.$
		\item [b)] Take $X_1 \in \mathcal{X}^W_{x_1}$ and $X_2 \in \mathcal{X}^W_{x_{2}}$ for some $x_1, x_2 > 0$ and $\alpha \in (0,1).$ 
		
		Then, noting that $\{\alpha X_1+(1-\alpha)X_2|X_1 \in \mathcal{X}^W_{x_1}, X_2 \in \mathcal{X}^W_{x_2}\}\subset \mathcal{X}^W_{\alpha x_1+(1-\alpha)x_2}$ one has
	\small
	\begin{flushleft}
		\begin{gather*}
		u^W_{c}(\alpha x_1+(1-\alpha)x_2)=\sup\limits_{X \in \mathcal{X}^W_{\alpha x_1+(1-\alpha)x_2}}\inf\limits_{Q \in \mathcal{Q}}E_{Q}[U_c^{W(\omega)}(X)] \\
		\geq
		\sup\limits_{\alpha X_1+(1-\alpha)X_2|X_1 \in \mathcal{X}^W_{x_1}, X_2 \in \mathcal{X}^W_{x_2}}\inf\limits_{Q \in \mathcal{Q}}E_{Q}[U_c^{W(\omega)}(\alpha X_1+(1-\alpha)X_2)]\\
		\geq 
		\sup\limits_{\alpha X_1+(1-\alpha)X_2|X_1 \in \mathcal{X}^W_{x_1}, X_2 \in \mathcal{X}^W_{x_2}}\inf\limits_{Q \in \mathcal{Q}}E_{Q}[\alpha U_c^{W(\omega)}(X_1)+(1-\alpha)U_c^{W(\omega)}(X_2)]\\
		\geq 
		\sup\limits_{\alpha X_1+(1-\alpha)X_2|X_1 \in \mathcal{X}^W_{x_1}, X_2 \in \mathcal{X}^W_{x_2}}[\alpha\inf\limits_{Q \in \mathcal{Q}}E_{Q} [U_c^{W(\omega)}(X_1)]+(1-\alpha)\inf\limits_{Q \in \mathcal{Q}}E_{Q}[U_c^{W(\omega)}(X_2)]]\\
		=\alpha \sup\limits_{X_1 \in \mathcal{X}^W_{x_1}}\inf\limits_{Q \in \mathcal{Q}}E_{Q}[U_c^{W(\omega)}(X_1)]+ (1-\alpha) \sup\limits_{X_2 \in \mathcal{X}^W_{x_2}}\inf\limits_{Q \in \mathcal{Q}}E_{Q}[U_c^{W(\omega)}(X_2)]\\
		=\alpha u^W_{c}(x_1)+ (1-\alpha) u^W_{c}(x_2).	\qedhere
		\end{gather*}
		\normalsize
	\end{flushleft}
	\end{itemize}
	
\end{proof}

\begin{lemma}\label[lemma]{ucW}
	Suppose that  \Cref{ass1} and \Cref{finit}  hold.
	
	Then, we have 
	\begin{gather*}
	u^W_{c}(x)=\sup\limits_{X \in \mathcal{X}^W_x}\inf\limits_{Q \in \mathcal{Q}}E_{Q}[U_c^{W(\omega)}(X)]=
	\inf\limits_{Q \in \mathcal{Q}}\sup\limits_{X \in \mathcal{X}^W_x}E_{Q}[U_c^{W(\omega)}(X)] \\
	=\sup\limits_{X \in \mathcal{X}^W_x}\inf\limits_{Q \in \mathcal{Q}_e}E_{Q}[U_c^{W(\omega)}(X)]=
	\inf\limits_{Q \in \mathcal{Q}_e}\sup\limits_{X \in \mathcal{X}^W_x}E_{Q}[U_c^{W(\omega)}(X)]. 
	\end{gather*}
\end{lemma}
\begin{proof}
	Take $\varepsilon>0.$ Consider $Y:=(X+\varepsilon)\wedge W,$ for $X \in \mathcal{X}^W_x.$ Then $Y \in \mathcal{X}^W_{x+\varepsilon}$, since $0 \leq Y \leq W$ and $\mathrm{E}_{Q^e}[Y]=\mathrm{E}_{Q^e}[(X+\varepsilon)\wedge W]\leq \mathrm{E}_{Q^e}[X+\varepsilon] \leq x+\varepsilon$ . 
	
	Define $Y^W_{X,\varepsilon}:=\{Y \in L^{1}(Q^e) \mid Y=(X+\varepsilon)\wedge W, X \in \mathcal{X}^W_x\}.$ Then $Y^W_{X,\varepsilon} \subset \mathcal{X}^W_{x+\varepsilon}.$
	Thus, it holds
	\begin{gather*}
	u^W_{c}(x+\varepsilon)= \sup\limits_{\bar{X} \in \mathcal{X}^W_{x+\varepsilon}}\inf\limits_{Q \in \mathcal{Q}}E_{Q}[U_c^{W(\omega)}(\bar{X})]\\
	\geq \sup\limits_{Y \in Y^W_{X,\varepsilon}}\inf\limits_{Q \in \mathcal{Q}}E_{Q}[U_c^{W(\omega)}(Y)] 
	=\sup\limits_{Y \in Y^W_{X,\varepsilon}}\inf\limits_{Q \in \mathcal{Q}}\mathbb{E}\left[U_c^{W(\omega)}(Y)\cdot \frac{dQ}{dP}\right]\\ 
	=\sup\limits_{Y \in Y^W_{X,\varepsilon}}\inf\limits_{Z \in \mathcal{Z}}\mathbb{E}[ZU_c^{W(\omega)}(Y)]. 
	\end{gather*}
	In the proof of \cite[Lemma 1]{Bul2019} is already shown that for each $X \in \mathcal{X}(x),$ the map $Z\mapsto \mathbb{E}[ZU_c^{W(\omega)}(Y)]$ is  a weakly lower semicontinuous affine functional defined on the weakly compact convex set $\mathcal{Z}.$
	
	Moreover, in the proof of \cite[Lemma 1]{Bul2019} is already shown
	that for each $Z\in \mathcal{Z}, X \rightarrow \mathbb{E}[ZU_c^{W(\omega)}(X+\varepsilon)]$ is a concave functional.
	Hence, one has that for each $Z\in \mathcal{Z}, X \rightarrow \mathbb{E}[ZU_c^{W(\omega)}(Y)]$ is a concave functional defined on the convex set $\mathcal{X}^W_{x}.$
	
	Noting that from the almost sure convergence follows weak convergence, the conditions of the lop sided minimax theorem \cite[Chapter 6, p. 295]{Aubin84} are satisfied, and so
	\begin{gather*}
	\sup\limits_{Y \in Y^W_{X,\varepsilon}}\min\limits_{Z \in \mathcal{Z}}\mathbb{E}[ZU_c^{W(\omega)}(Y)]=
	\min\limits_{Z \in \mathcal{Z}}\sup\limits_{Y \in Y^W_{X,\varepsilon}}\mathbb{E}[ZU_c^{W(\omega)}(Y)].
	\end{gather*}
	
	Hence, we arrive at
	\begin{gather*}
	u^W_{c}(x+\varepsilon)
	\geq \sup\limits_{Y \in Y^W_{X,\varepsilon}}\inf\limits_{Z \in \mathcal{Z}}\mathbb{E}[ZU_c^{W(\omega)}(Y)]
	=\min\limits_{Z \in \mathcal{Z}}\sup\limits_{Y \in Y^W_{X,\varepsilon}}\mathbb{E}[ZU_c^{W(\omega)}(Y)]\\
	\geq
	\inf\limits_{Q \in \mathcal{Q}}\sup\limits_{Y \in Y^W_{X,\varepsilon}}E_{Q}[U_c^{W(\omega)}(Y)]\\
	\overset{Y\geq X}{\geq}
	\inf\limits_{Q \in \mathcal{Q}}\sup\limits_{Y \in Y^W_{X,\varepsilon}}E_{Q}[U_c^{W(\omega)}(X)]
	\geq \sup\limits_{X \in \mathcal{X}^W_x}\inf\limits_{Q \in \mathcal{Q}}E_{Q}[U_c^{W(\omega)}(X)]=u^W_c(x).
	\end{gather*}
	
	The last inequality follows from the fact that for all $Q \in \mathcal{Q}$ and $X \in \mathcal{X}^W_x$
	\begin{gather*}
	\sup\limits_{X \in \mathcal{X}^W_x}E_Q[U_c^{W(\omega)}(X)]\geq \inf\limits_{Q \in \mathcal{Q}}E_Q[U_c^{W(\omega)}(X)].
	\end{gather*} 
	
	Sending $\varepsilon \downarrow 0$ and using the continuity of $u^W_{c},$ as a concave function on set $(0, +\infty)$, we obtain  the first part of the lemma.
	
	From \Cref{finit} and \cite[Lemma 3.3]{Schied-Wu} follows that $u^W_{c}(x)=\inf\limits_{Q \in \mathcal{Q}_e}u^W_{c,Q}(x).$ (the proof is similar to the proof of \cite[Lemma 2]{Bul2019}).
	
	Hence, 
	\begin{gather*}
	u^W_{c}(x)=\inf\limits_{Q \in \mathcal{Q}_e}u^W_{c,Q}(x)=
	\inf\limits_{Q \in \mathcal{Q}_e}\sup\limits_{X \in \mathcal{X}_x^W}E_{Q}[U_c^{W(\omega)}(X)]\\
	\geq \sup\limits_{X \in \mathcal{X}_x^W}\inf\limits_{Q \in \mathcal{Q}_e}E_{Q}[U_c^{W(\omega)}(X)]\geq
	\sup\limits_{X \in \mathcal{X}_x^W}\inf\limits_{Q \in \mathcal{Q}}E_{Q}[U_c^{W(\omega)}(X)]=u^W_{c}(x).
	\end{gather*}
	
	Which concludes the proof.
\end{proof}
\subsection{Minimax identity for the objective function $U(x)$}

In this section, we will establish auxiliary results which will allow us to complete the proof of \autoref{mt2}.

\begin{lemma}\label[lemma]{lemma5_2}
	If \Cref{ass1} and \Cref{finiteu} hold,
	then for all $X \in \mathcal{X}_x^W$ 
	\begin{gather}
	\inf\limits_{Q \in \mathcal{Q}}E_Q[U^{W}(X)] = 	\inf\limits_{Q \in \mathcal{Q}_e}E_Q[U^W(X)].
	\label{6starW}
	\end{gather}
\end{lemma}	

\begin{proof}
	The proof is the same as in a non-constrained case. See \cite[Lemma 2]{Bul2019}.
\end{proof}

\begin{lemma}\label[lemma]{lemmaprop1}\hypertarget{lemmaprop1}{}	
	Let $\{P_v, v\in (0,\infty)\}$ be a family of atomless probability measures on a standard Borel space $(\Omega, \mathcal F)$, such that for any $A\in \mathcal F$, $P_\cdot(A)$ is measurable. 
	Then, for all $Q \in \mathcal{Q}_e$, for all $X \in \mathcal{X}^W_x$ there exists $X^{\star} \in \mathcal{X}^W_x$ such that
	\begin{gather}
	E_Q[U^{W}(X^{\star})] = E_Q[U_c^{W}(X^{\star})]=E_Q[U_c^{W}(X)]\geq E_Q[U^{W}(X)].
	\label{eqprop1}
	\end{gather}
\end{lemma}	

\begin{proof}
	The main idea of the proof is to utilize the ideas of \cite[Proposition 5.3]{Reichlin11} in our conditional setting. 
	
	Fix $Q\in \mathcal{Q}_e$ and define $\psi = dQ^e/d\mathbb P$, $\varphi = dQ^e/dQ$. First of all, note that for any $Q\in \mathcal{Q}_e$, there exists a corresponding regular conditional probability given $W$. Indeed, since $\psi$ is positive, $M(v) :=  \int_{\Omega} \psi\, P(d\omega,v)$ is positive as well so
	$$
	P_{Q}(A,v) = \frac{\int_{A} \psi\, P(d\omega,v)}{M(v)}, A\in\mathcal{F},
	$$
	is a probability measure. It is easy to see that it is measurable in $v$ and $P_Q(A,W) = Q(A\mid W)$ \ $Q$-a.s.

	By \Cref{measurability} applied to $Y(x,\omega) = X(\omega)$, $\phi(v,\omega) = \varphi(\omega)$ and $P_v(A) = P_Q(A,v)$, there exists a jointly measurable function $Y^\star(v,\omega)$ such that for all $v>0$, $E_{P_Q(\cdot, v)}[Y^\star(v,\omega)\varphi(\omega)]\le E_{P_Q(\cdot,v)}[X(\omega)\varphi(\omega)]$ and
	\begin{gather*}
	E_{P_Q(\cdot,v)}[U^{v}\big(Y^{\star}(v,\omega)\big)]= E_{P_Q(\cdot,v)}[U_c^{v}\big(Y^{\star}(v,\omega)\big)]=E_{P_Q(\cdot,v)}[U_c^{v}\big(X(\omega)\big)]. 
	\end{gather*}
	Set $X^\star(\omega) = Y^\star(W(\omega),\omega)$. Then
	\begin{gather*}
	E_{Q^e}\left[X^{\star}\right]
	=E_{Q}\left[Y^{\star}\big(W(\omega),\omega\big)\varphi\right]
	=E_{Q}\left[E_{Q}\left[Y^{\star}\big(W(\omega),\omega\big)\varphi\mid W\right]\right] \\
	=E_{Q}\left[E_{P_Q(\cdot, v)}[Y^\star(v,\omega)\varphi(\omega)] \big|_{v=W}\right]\le E_{Q}\left[E_{P_Q(\cdot, v)}[X(\omega)\varphi(\omega)] \big|_{v=W}\right] \\
	= E_{Q}\left[E_{Q}\left[X(\omega)\varphi\mid W\right]\right] = E_{Q}\left[X(\omega)\varphi\right] = E_{Q^e}[X]\le x,
	\end{gather*}
	so $X^{\star} \in \mathcal{X}_x^W$. Further,
	\begin{gather*}
	E_Q\left[U_c^{W}(X^{\star})\right]
	=E_{Q}\left[U_c^{W(\omega)}\big(Y^{\star}(W(\omega),\omega)\big)\right]
	\\=E_{Q}\left[E_{Q}\left[U_c^{W(\omega)}\big(Y^{\star}(W(\omega),\omega)\big)\mid W\right]\right] \\
	=E_{Q}\left[E_{P_Q(\cdot, v)}[U_c^{v}\big(Y^{\star}(v,\omega)\big)] \big|_{v=W}\right]= E_{Q}\left[E_{P_Q(\cdot, v)}[U_c^{v}\big(X(\omega)\big)]\big|_{v=W}\right] \\
	= E_{Q}\left[E_{Q}\left[U_c^{v}\big(X(\omega)\big)\mid W\right]\right] = E_{Q}\left[U_c^{v}\big(X(\omega)\big)\right].
	\end{gather*}
	The equality $E_Q[U^{W}(X^{\star})] = E_Q[U_c^{W}(X^{\star})]$ is proved similarly, and the inequality $E_Q[U_c^{W}(X)]\geq E_Q[U^{W}(X)]$ is obvious, since $U_c^W\ge U^W$. The proof is now complete. 
\end{proof}

\begin{lemma}\label[lemma]{U=Uc}\hypertarget{lemma}{}
	If \Cref{atomless} holds, then for all $Q \in \mathcal{Q}_e$ it holds that
	\begin{gather*}
	\sup\limits_{X \in \mathcal{X}_x^W}E_Q[U^{W(\omega)}(X)] = \sup\limits_{X \in \mathcal{X}_x^W}E_Q[U_c^{W(\omega)}(X)], \text{ for all } x>0. 
	\end{gather*}
\end{lemma}
\begin{proof}
	Apply the $\sup\limits_{X \in \mathcal{X}_x^W}$ to the all parts of \eqref{eqprop1}. Then one has
	\begin{gather}
	\sup\limits_{X \in \mathcal{X}_x^W}E_Q[U^{W(\omega)}(X^{\star})] 
	= \sup\limits_{X \in \mathcal{X}_x^W}E_Q[U_c^{W(\omega)}(X^{\star})] \nonumber\\
	=\sup\limits_{X \in \mathcal{X}_x^W}E_Q[U_c^{W(\omega)}(X)]\geq \sup\limits_{X \in \mathcal{X}_x^W}E_Q[U^{W(\omega)}(X)].
	\label{eqprop2}
	\end{gather}
	Since $Q \in \mathcal{Q}_e$ is arbitrary, $X \in \mathcal{X}^W_x$ is arbitrary and $X^{\star} \in \mathcal{X}^W_x$ it follows that the inequality in \eqref{eqprop2} is an equality and, hence, the statement of the lemma is proven. 
\end{proof}


\begin{lemma}\label[lemma]{U=Uc2}
	Under the \Cref{atomless}, 
	\begin{gather*}
	\inf\limits_{Q \in \mathcal{Q}_e}\sup\limits_{X \in \mathcal{X}_x^W}E_Q[U^{W(\omega)}(X)] = \inf\limits_{Q \in \mathcal{Q}_e}\sup\limits_{X \in \mathcal{X}_x^W}E_Q[U_c^{W(\omega)}(X)].
	\label{infsup2}
	\end{gather*}
\end{lemma}	

\begin{proof}
	Follows immediately from the \Cref{U=Uc}.
\end{proof}

\begin{proof}[Proof of the \autoref{mt2}] \mbox{}\\*
	\begin{itemize}
		\item \hyperlink{star2}{$(1\star)$} - \hyperlink{star2}{$(3\star)$} follows from the \Cref{ucW};
		
		\item \hyperlink{star2}{$(4\star)$} follows from the fact that $U_c^{W(\omega)} \geq U^{W(\omega)}$;
		
		\item \hyperlink{star2}{$(5\star)$} follows from the \Cref{U=Uc2};
		
		\item To obtain \hyperlink{star2}{$(6\star)$} we need to take the $\sup\limits_{X \in \mathcal{X}_x^W}$ of the both sides in the equality \eqref{6starW};
		
		\item The inequality \hyperlink{star2}{$(7\star)$} follows from the fact that for all $Q \in \mathcal{Q} $ and all $X \in \mathcal{X}_x^W$ holds 
		\begin{gather*}
		\inf\limits_{Q \in \mathcal{Q}}E_Q[U^{W(\omega)}(X)] \leq \sup\limits_{X \in \mathcal{X}_x^W}E_Q[U^{W(\omega)}(X)].
		\end{gather*}
		
		\item Since $\mathcal{Q}_e  \subseteq \mathcal{Q}, $ the inequality \hyperlink{star2}{$(8\star)$} is clear.
	\end{itemize}
\end{proof}

\begin{appendix}

\section{Auxiliary statements}
In what follows $U\colon \bbR_+ \to \bbR_+$ is a non-decreasing upper-semicontinuous function satisfying a mild growth condition, $U^v(y) = U(y\wedge v), v>0$, and $U^v_c$ is the concavification of $U^v$. For $v>y>0$, let
$$
a(v,y) = \begin{cases}
\inf\{z\le y: U^v_c(x)> U^v(x)\text{ on }[z,y]\}, &U^v(y)< U^v_c(y),\\
y, & U^v(y) = U^v_c(y)
\end{cases}
$$
and 
$$
b(v,y) = \begin{cases}
\sup\{z\le y: U^v_c(x)> U^v(x)\text{ on }[y,z]\}, &U^v(y)< U^v_c(y),\\
y, & U^v(y) = U^v_c(y)
\end{cases}
$$
be the left and right endpoints of the interval around $y$, in which where $U^v < U^v_c$ (or just $y$ in the case where $U^v(y) = U^v_c(y)$). Observe that $U^v(a(v,y)) = U^v_c(a(v,y))$ and $U^v(b(v,y)) = U^v_c(b(v,y))$: in the case of inequality we would have it in some open interval, contradicting the definition of infimum or supremum.

\begin{lemma}\label{semicont}\hypertarget{semicont}{}
	The functions $a$ and $b$ defined above are measurable. 
\end{lemma}
\begin{proof}
	We will show only measurability of $a$, that of $b$ can be shown similarly. 
	
	Note that $a$ is obviosly non-decreasing in $y$. It is also right-continuous in $y$. Indeed, let $y_n\ge y_0$, $y_n\to y_0$, $n\to\infty$. If $U^v(y_0) < U^v_c(y_0)$, then, thanks to continuity of $U^v_c$ and upper-semicontinuity of $U^v$, this inequality holds in an open interval around $y_0$, which means that $a(v,y_n) = a(v,y_0)$ for all $n$ large enough. Otherwise, if $U^v(y_0) = U^v_c(y_0)$, then $a(v,y_n)\in (y_0,y_n]$ for all $n\ge 1$, whence $a(v,y_n)\to y_0 = a(v,y_0), n\to\infty$. 
	
	Further, since for $v_1< v_2$, $U^{v_2}_c$ dominates $U^{v_1}$ on $[0,v_1]$, we have that $U^{v_2}_c \ge U^{v_1}_c$. Consequently, $a$ is non-increasing in $v$. Now the proof follows from the following lemma.
	\begin{lemma}
		Let a function $f\colon (0,\infty)^2\to \bbR$ be such that for each $x> 0$, $f(x,\cdot)$ is non-decreasing and right-continuous, and for each $y> 0$, $f(\cdot,y)$ is non-decreasing. Then, $f$ is measurable.
	\end{lemma}	
	\begin{proof}
		For arbitrary $t\in\bbR$ consider the set $A_t = f^{-1}((-\infty,t))$. Thanks to monotonicity, $(x,y)\in A_t \Rightarrow (x',y') \in A_t$ for all $x'\le x, y'\le y$. Moreover, thanks to right-continuity in $y$, the $x$-sections $A_{t,x} = \{y> 0: (x,y) \in A_t\}$ are open intervals. 
		
		Define $A_{t,x+} = \bigcup_{z>x} A_{t,z}$. We claim that the set $A^o_t := \{(x,y)\in (0,\infty)^2: y\in A_{t,x+}\}$ is open (it is actually the interior of $A_{t}$). Indeed, if $(x,y) \in A_0^t$, then $y\in A_{t,z}$ for some $z>x$. Since $A_{t,z}$ is open, for some $\varepsilon>0$, $(y-\varepsilon, y+\varepsilon) \subset A_{t,z}$. Then, thanks to monotonicity, $(0,z)\times (y-\varepsilon, y+\varepsilon) \subset A_{t}^0$. 
		
		By the definition of $A_t^o$,
		$$
		A_t\setminus A_t^o = \bigcup_{x>0} \{x\}\times (A_{t,x}\setminus A_{t,x+}).
		$$
		For any $x>0$, $A_{t,x}\setminus A_{t,x+}$ is a difference of two open intervals, so it's either a half-open interval or empty. Since the half-open intervals for different $x$ are disjoint, there are at most countable number of then. Therefore, $A_t\setminus A_t^o$ is Borel as a countable union of Borel sets, which finishes the proof.		
	\end{proof}
	\renewcommand{\qedsymbol}{}
\end{proof}


\begin{lemma}\label{uniformrv}\hypertarget{uniformrv}{}
	Let $\{P_v, v \in (0,\infty) \}$ be a family of atomless probability measures on a standard Borel space $(\Omega, \mathcal F)$, such that for any $A\in \mathcal F$, $P_\cdot(A)$ is a measurable, and $\xi\colon (0,\infty)\times \Omega\to\bbR$ be measurable. Then, there exist measurable functions $\zeta\colon (0,\infty)\times \Omega \to \mathbb{R}$ and $q\colon (0,\infty)\times \bbR\to \bbR$ such that for all $v\in (0,\infty)$, $\zeta(v,\cdot)$ has a uniform distribution on $(0,1)$ with respect to $P_v$, $q(v,\cdot)$ is non-decreasing, and $q(v,\zeta(v,\omega)) = \xi(v,\omega)$ \ $P_v$-a.s. 
\end{lemma}
\begin{proof}
	Since $(\Omega, \cF)$ carries an atomless measure, it is uncountable. Then it is well known that it is isomorphic to $(\bbR, \cB(\bbR))$, i.e. there exists a measurable bijection $\tau\colon \Omega \to \bbR$ such that $\tau^{-1}$ is measurable as well. Therefore we can assume without loss of generality that $(\Omega, \cF) = \big((0,1), \cB((0,1))\big)$. 
	
	Assume first that the distribution of $\xi(v,\omega)$ is continuous for all $v\in (0,\infty)$. The cumulative distribution function $F_\xi(v,x) = P_v(\{\xi(v,\omega)\le x\})$ is jointly measurable (see \eg \cite[Lemma 4.1]{shevchenko}), so the quantile function $q_\xi(v,r) = \inf\{x\in\bbR: F_\xi(v,x)\ge r\}$ is jointly measurable as well. So in this case we can set $\zeta(v,\omega) = F_\xi(v,\xi(v,\omega))$ and $q(v,r) = q_\xi(v,r)$; by the quantile transformation theorem, $\zeta$ and $q$ are as required. 
	
	For general $\xi$, define 
	$$
	\kappa(v,x) = P_v(\{\omega\colon \xi(v,\omega)<x\}) + P_v(\{\omega\le x: \xi(v,\omega) = \xi(v,x)\}), x\in (0,1),
	$$
	which is jointly measurable thanks to \cite[Lemma 4.1]{shevchenko}. It is easy to see that for any $v\in (0,\infty)$, $\kappa$ has continuous distribution under $P_v$, and 
	$$
	q_\xi(v,\kappa(v,\omega)) = \xi(v,\omega),
	$$
	where, as above, $q_\xi$ is the quantile function of $\xi$. Then we can set $\zeta(v,\omega) = F_\kappa(v,\kappa(v,\omega))$ and $q(v,r) = q_\xi(v,q_\kappa(v,r))$, arriving at the desired statement.	
\end{proof}

\begin{lemma}\label{measurability}\hypertarget{measurability}{}
	Let $\{P_v, v\in (0,\infty)\}$ be a family of atomless probability measures on a standard Borel space $(\Omega, \mathcal F)$, such that for any $A\in \mathcal F$, $P_\cdot(A)$ is measurable. Also let $Y,\phi:(0,\infty) \times \Omega\to [0,\infty)$ be jointly measurable functions such that $Y(v,\omega)\le v$ for all $v>0$, $\omega\in \Omega$.
	Then, there exists a jointly measurable function $Y^\star(v,\omega)$ such that for all $v>0$, $E_{P_v}[Y^\star(v,\omega)\phi(v,\omega)]\le E_{P_v}[Y(v,\omega)\phi(v,\omega)]$ and
	\begin{gather*}
	E_{P_v}[U^{v}\big(Y^{\star}(v,\omega)\big)]= E_{P_v}[U_c^{v}\big(Y^{\star}(v,\omega)\big)]=E_{P_v}[U_c^{v}\big(Y(v,\omega)\big)]. 
	\end{gather*}
\end{lemma}
\begin{proof}
	We will adapt the construction used in the proof of \cite[Proposition 5.3]{Reichlin11} so that it has the desired measurability property.

	Define
	\begin{gather*} 
	S=\{(v,\omega) \in (0,\infty)\times \Omega: U^v(Y(v,\omega))<U_c^v(Y(v,\omega))\}
	\end{gather*}
	and for $(v,\omega) \in S$, let
	\begin{gather*}
	\alpha(v,\omega):=\inf\{z:U^v(x)<U_c^v(x) \text{ on } (z,X(\omega)]\} \label{a(x,V)}, \\ 
	\beta(v,\omega):=\sup\{z:U^v(x)<U_c^v(x) \text{ on } [X(\omega), z)\}. \label{b(x,V)}
	\end{gather*}
	be the left and the right ends of the interval where $U^v<U_c^v$. These functions are measurable thanks to \Cref{semicont}. 
	
	For $(v,\omega)\in S$, define
	\begin{gather*}
	\lambda(v,\omega) = \frac{\beta(v,\omega) - X(\omega)}{\beta(v,\omega) - \alpha(v,\omega)}
	\end{gather*}
	so that $X(v,\omega) = \lambda(v,\omega) \alpha(v,\omega) + (1-\lambda(v,\omega))\beta(v,\omega)$. Due to \Cref{uniformrv}, there exist measurable functions $\zeta,q\colon (0,\infty)\times \Omega \to \bbR$ such that for all $v>0$, $\phi(v,\omega) = q(v,\zeta(v,\omega)) $ $ P_v$-a.s. and $\zeta(v,\omega)$ is uniformly distributed on $(0,1)$ under $P_v$. 
	For $s\in[0,1], v>0$ and $\omega,\omega'\in\Omega$, define
	\begin{gather*}
	h(s,v,\omega,\omega') = \bbI_{(v,X(\omega))\in S,(v,X(\omega'))\in S, a(v,X(\omega))= a(v,X(\omega'))}\big(\bbI_{\zeta(v,\omega')<s} - \lambda(v,\omega')\big)
	\end{gather*}
	and
	\begin{gather*}
	f(s,v,\omega) = \int_{\Omega} h(v,\omega,\omega',s) P_v(d\omega').
	\end{gather*}
	Since $\lambda(v,\omega)\in (0,1)$ and $\zeta$ has continuous distribution under $P_v$, we have that for all $(v,\omega)\in S$, $f$ is continuous in $s$ and $f(0,v,\omega)<0<f(1,v,\omega)$. Denoting $\sigma(v,\omega) = \inf\{s\in(0,1): f(s,v,\omega)\ge 0\}$, we have $f(\sigma(v,\omega),v,\omega) = 0$. Also for any $s\in (0,1)$, 
	$\{(v,\omega): \sigma(v,\omega)\le s\} = \{(v,\omega): f(s,v,\omega)\ge 0\}$, so $\sigma(v,\omega)$ is measurable. 
	
	Now set
	\begin{gather*}
	Y^{\star}(v,\omega)=
	\begin{cases}
	Y(v, \omega),  &(v,\omega) \notin S; \\
	\alpha(v, \omega),  &(v,\omega) \in S\cap \{\zeta(v,\omega) < \sigma(v,\omega)\}; \\
	\beta(v, \omega), &(v,\omega) \in S\cap \{\zeta(v,\omega) \ge  \sigma(v,\omega)\}. 
	\end{cases}
	\end{gather*}
	Since for any fixed $v>0$, the construction coincides with that given in \cite[Proposition 5.3]{Reichlin11}, the rest of the proof follows. 
\end{proof}

\end{appendix}

\begin{acknowledgement}[title={Acknowledgments}]
O. Bahchedjioglou thanks Prof. Dr. Mitja Stadje and Dr. Thai Nguyen for their help and support
during her work on the topic.
\end{acknowledgement}
%

\bibliographystyle{bib/vmsta-mathphys}
\bibliography{bib/biblio}

\end{document}